\documentclass[11pt]{amsart} 
\usepackage{amsmath,amssymb,amsthm,mathrsfs}
\newtheorem{theorem}{Theorem}[section]
\newtheorem{proposition}[theorem]{Proposition}
\newtheorem{corollary}[theorem]{Corollary}

\theoremstyle{definition}
\newtheorem{definition}[theorem]{Definition}

\newtheorem{remark}[theorem]{Remark}
\usepackage[english]{babel}
\usepackage[utf8]{inputenc}
\usepackage{xcolor}
\usepackage{multirow} 
\usepackage{color}

\usepackage{graphicx}
\usepackage[left=3cm,right=3cm]{geometry}

\def\r{\mathbb R}\def\s{\mathbb S}
\def\H{\mathsf H}
 
\def\h{\mathbb H}

\begin{document}

\title{Horo-shrinkers in the hyperbolic space}
\author{Antonio Bueno, Rafael L\'opez}
\address{ Departamento de Ciencias\\  Centro Universitario de la Defensa de San Javier. 30729 Santiago de la Ribera, Spain}
\email{antonio.bueno@cud.upct.es}
\address{ Departamento de Geometr\'{\i}a y Topolog\'{\i}a\\  Universidad de Granada. 18071 Granada, Spain}
\email{rcamino@ugr.es}
\begin{abstract}
A surface $\Sigma$ in the hyperbolic space $\h^3$ is called a horo-shrinker if its mean curvature $H$ satisfies $H=\langle N,\partial_z\rangle$, where $(x,y,z)$ are the coordinates of $\h^3$ in the upper half-space model and $N$ is the unit normal of $\Sigma$. In this paper we study horo-shrinkers invariant by one-parameter groups of isometries of $\h^3$ depending if these isometries are hyperbolic, parabolic or spherical. We characterize totally geodesic planes as the only horo-shrinkers invariant by a one-parameter group of hyperbolic translations. The grim reapers are defined as the horo-shrinkers invariant by a one-parameter group of parabolic translations. We describe the geometry of the grim reapers proving that   they are periodic surfaces. In the last part of the paper,  we give a complete classification of horo-shrinkers invariant by spherical rotations, distinguishing if the surfaces intersect or not the rotation axis.
\end{abstract}



\maketitle

\section{Introduction }\label{sec1}

The theory of the mean curvature flow (MCF for short) is an area of great activity in geometric analysis in the last decades: see, for example, the surveys \cite{co,ec,ma} and references therein. In Euclidean space $\r^3$, let be $\Sigma$ an oriented smooth surface and $\Psi\colon\Sigma\to\r^3$ an isometric immersion. A MCF for $\Psi$ is a smooth   family of immersions $\{\Psi_t:\Sigma\to\r^3: t\in [0,T)\}$ satisfying 
$$\left\{\begin{array}{ll}
\displaystyle\frac{\partial\Psi_t}{\partial t}  &=H(\Psi_t)N(\Phi_t),\\
\Psi_0&=\Psi,
\end{array}\right.$$
where  $H(\Phi_t)$ and $N(\Phi_t)$ are the mean curvature and the unit normal of $\Phi_t$ respectively.  Solutions of the MCF develop singularities at finite time, which may   cause a change in the topology of the surface. There are two types of singularities. Translators of the MCF (also called translating solitons) appear as the equation of the limit flow by a blow-up procedure near type II singularities, according to Huisken and Sinestrari \cite{hs}. In $\r^3$, a translator is a surface $\Sigma$ characterized by the equation $H=\langle N,\vec{v}\rangle$, where $H$ and $N$ are the mean curvature and the unit normal of $\Sigma$, respectively, and $\vec{v}$ is a direction of the ambient space. This direction $\vec{v}$ indicates that the shape of $\Sigma$ does not change  during the evolution because $\Sigma$  is translated   by the MCF  at constant velocity.

Huisken initiated the study of the MCF in general Riemannian manifolds \cite{hu}. When the ambient space is the hyperbolic space, pioneering research on the MCF is  \cite{cm1,cm2}; see also \cite{an,aw,cm3,ge}. Singularities of MCF have been less studied \cite{ji,wa}. Nevertheless, the same notion of translator can be defined by replacing $\vec{v}\in\r^3$ by a Killing vector field $X\in\mathfrak{X}(\h^3)$ whose flow of isometries consists of translations of $\h^3$. A surface $\Sigma\subset\h^3$ is a translator   with respect to $X$ if $H=\langle N,X\rangle$. In the hyperbolic space $\h^3$ there are two types of translations \cite{dd}. Parabolic translations   are isometries  of $\h^3$ that  fix one point of the ideal boundary $\h^3_\infty$. Hyperbolic translations are  isometries of $\h^3$ that fix two points of $\h^3_\infty$. The corresponding translators have been recently studied in \cite{bl} and \cite{li}, respectively.

 Besides Killing vector fields,   another vector fields of $\h^3$   of special relevance are the   conformal vector fields.   
  In order to give explicit example of such vector fields,  we will use   the upper half-space model of $\h^3$, that is, $(\r^3_{+},\langle\cdot,\cdot\rangle)$, where $\r_+^3=\{(x,y,z)\in\r^3:z>0\}$, $\langle\cdot,\cdot\rangle$ is  the hyperbolic metric 
$$ \langle\cdot,\cdot\rangle= \frac{1}{z^2}\langle\cdot,\cdot\rangle_e,$$
and $\langle\cdot,\cdot\rangle_e=dx^2+dy^2+dz^2$ is the Euclidean metric of $\r^3$. In this model, the vector field $\partial_z\in \mathfrak{X}(\h^3)$ is a conformal vector field because its Lie derivative is $\mathcal{L}_{\partial_z}\langle,\rangle=-\frac{2}{z}\langle,\rangle$. Motivated by the notion of translator in $\h^3$, we give the following definition. 

\begin{definition} A horo-shrinker in $\h^3$ is an isometric immersion $\Psi\colon\Sigma\to\h^3$ of an oriented smooth surface $\Sigma$ whose mean curvature $H$ satisfies 
\begin{equation}\label{eq1}
H=\langle N,\partial_z\rangle,
\end{equation}
where $N$ is the unit normal of $\Sigma$.
\end{definition}

It was in \cite{alr} where the authors proposed the study of self-similar solutions of the mean curvature flow in the presence of a conformal  vector field. More precisely, the conformal vector field $\partial_z$ corresponds with the vector field $e^t\partial_t$  in Example 3.3 of \cite{alr} when the hyperbolic space $\h^3$ is viewed as the warped product $\r\times_{e^t}\r^2$.   This was motivated by the examples of self-shrinkers of the mean curvature flow in $\r^3$, where the vector field  is the position vector field.    In fact, horo-shrinkers  are the analogues of the self-shrinkers  of the MCF theory in Euclidean space $\r^3$ \cite{co,ec,ma}. Instead the vector field $\partial_z$, in \cite{mr} the authors considered the conformal vector field $-\partial_z$ and the corresponding solutions of \eqref{eq1} are called horo-expanders. In $\r^3$,   self-shrinkers and self-expanders are surfaces which move by homotheties (contractions or expansions, respectively) when they evolve by the MCF.  In contrast to the translators of $\h^3$, the shape of   horo-shrinkers and horo-expanders is not  preserved along the MCF.      Another reason to consider surfaces satisfying \eqref{eq1} is because of its formal similarity with Ricci solitons \cite{lo,ma}.

A first  observation is that horo-shrinkers  and horo-expanders  are minimal surfaces in the sense of Ilmanen \cite{il}.   Specifically, if we define the function $\phi(x,y,z)=-2/z$, then a minimal surface for the conformal metric $e^\phi \langle,\rangle$ is characterized by Eq. \eqref{eq1}, that is, the surface is a   horo-shrinker. In case of horo-expanders the function is   $\phi(x,y,z)=2/z$. Recall that being minimal in a conformal space is a property that also fulfill self-shrinkers and self-expanders of $\r^3$. However, and in contrast to the theory of the MCF in Euclidean space, it is unknown if the translators of $\h^3$ determined by parabolic and hyperbolic translations and described in \cite{bl,li} are minimal surfaces in the sense of Ilmanen. 

The purpose of  this paper is to investigate the horo-shrinkers. To be precise, we are interested   in those horo-shrinkers that are invariant by a one-parameter group of isometries of $\h^3$.  We detail the organization of the paper and highlight some of the main results. In   Section \ref{sec2} we show the first examples of horo-shrinkers, such as vertical planes (totally geodesic planes) and the horosphere $\H_1$ of equation $z=1$. Taking these examples as comparison surfaces and by the tangency principle, we prove that   there are no closed horo-shrinkers. We also classify in Thm.  \ref{t1} all   horo-shrinkers invariant by hyperbolic translations. In Section \ref{sec3} we define the grim reapers as those horo-shrinkers invariant by a one-parameter group of parabolic translations. The classification of the grim reapers is given in Thm.  \ref{t2}, being these surfaces   vertical planes, the horosphere $\H_1$, and a one-parameter family of periodic surfaces along a horizontal direction orthogonal to the parabolic translations. As a consequence of the properties of the grim reapers, we will prove in Thm.  \ref{t22} that there are no solutions of the Dirichlet problem at infinity associated to the non-parametric equation for \eqref{eq1}.

Section \ref{sec4}  is devoted to rotational horo-shrinkers of spherical type. We distinguish if the surfaces intersect or not the rotation axis. In the first case, the existence of these surfaces is not a direct consequence of standard theory, since the ODE fulfilled is degenerated when the surface intersects  the rotation axis. In Thm.  \ref{t3} we prove such existence using Banach's fixed point theorem. In Thm.  \ref{t4} we prove  that they are parametrized by one parameter, namely,  the initial height at which they intersect the rotation axis. We also prove that these surfaces oscillate around $\H_1$ converging to it. Finally, in Thm.  \ref{t5} we describe the spherical rotational horo-shrinkers that do not intersect the rotation axis. These surfaces are parametrized by two parameters and they  oscillate around $\H_1$.

\section{Preliminaries}\label{sec2}

In this paper we use the upper half-space model of $\h^3$. We will  employ the terminology  parallel  in the Euclidean sense and  by  vertical  and   horizontal  we mean to be  parallel to the $z$-axis or parallel to the $xy$-plane, respectively. The ideal boundary $\h_\infty^3$  of $\h^3$ is the one-point compactification of the plane $z=0$.  We show some explicit examples of horo-shrinkers.

\begin{enumerate}
\item Vertical totally geodesic planes. These surfaces are minimal ($H=0$) and the  unit normal $N$ is orthogonal to $\partial_z$. 
\item The horosphere of equation $z=1$. This horosphere will be denoted by $\H_1$. In general,  horospheres in the upper half-space model can be viewed as horizontal planes of equation $z=c$, $c>0$. The mean curvature is $H=1$ with the orientation $N=c\partial_z$. Then $\langle N,\partial_z\rangle=1/c$ and consequently, the only horosphere $z=c$ satisfying \eqref{eq1} is when $c=1$.
\end{enumerate}

Next, we express the condition of being a horo-shrinker in a non-parametric way. For this,  we will use a  relation between the hyperbolic mean curvature $H$ of a surface $\Sigma$ in $\h^3$ and its Euclidean mean curvature $H_e$ when $\Sigma$ is regarded as a surface in $(\r^3_+,\langle,\rangle_e)$.  This relation is given by  
\begin{equation}\label{mean}
H(x,y,z)=zH_e(x,y,z)+(N^e)_3(x,y,z),\quad (x,y,z)\in\Sigma,
\end{equation}
where $N^e$ is the Euclidean unit normal   of $\Sigma$ and the subindex $(\cdot)_3$ denotes the third coordinate of the vector. From the viewpoint of PDE theory, equation \eqref{eq1} is of second order and  elliptic. Indeed, if  $\Sigma$ is locally expressed as $z=u(x,y)$, in virtue of \eqref{mean}, equation \eqref{eq1} writes as 
\begin{equation}\label{div}
\mbox{div}\left(\frac{Du}{\sqrt{1+|Du|^2}}\right)=\frac{2}{ \sqrt{1+|Du|^2}}\frac{1-u}{u^2}.
\end{equation}
This elliptic equation is of quasilinear  type.  As stated in the introduction,  horo-shrinkers are minimal surfaces in the conformal space $(\h^3,e^{-2/z} \langle,\rangle)$. This allows to use the tangency principle similarly as for minimal surfaces of $\r^3$. This implies that   if two horo-shrinkers have a common tangent point and one horo-shrinker locally lies at one side of the other around that point, then both horo-shrinkers agree in an open set. A first consequence of the tangency principle is a certain control of the  height of the points of a horo-shrinker that are critical points of height function.   

\begin{proposition}\label{mm}
 Let $\Sigma$ be a horo-shrinker,  $\Sigma\not=\H_1$. If $p\in\Sigma$ is a local maximum (resp. minimum) of the function $z\colon\Sigma\to\r$, then $z(p)>1$ (resp. $z(p)<1$).
\end{proposition}

\begin{proof} Let $p=(x_0,y_,z_0)$ be a local maximum of the function $z$. Since $Du(x_0,y_0)=(0,0)$, Eq. \eqref{div} becomes simply
$$\Delta u(x_0,y_0)=2\frac{1-z_0}{z_0^2},$$
where $\Delta$ is the Euclidean Laplacian of $\r^2$. Thus $\Delta u(x_0,y_0)\leq 0$ implies $z_0\geq 1$. We prove that, in fact, $z_0>1$. On the contrary, if $z_0=1$, then the horo-shrinker $\Sigma$ lies in one side of $\H_1$ in an open set of $p$. Then  the tangency principle would imply $\Sigma\subset\H_1$, which it is a contradiction. In case that $p$ is a local minimum, the arguments are analogous.
\end{proof}

A second  consequence of the tangency principle is the following result.

\begin{proposition} There are no closed (compact without boundary) horo-shrinkers. 
\end{proposition}

\begin{proof} Arguing by contradiction, assume that $\Sigma$ is a closed horo-shrinker. Take $\Pi$   a vertical plane of $\r^3_{+}$  that does not intersect $\Sigma$. We move $\Pi$ towards $\Sigma$ until we arrive to a first contact point between both surfaces. Since $\Pi$ is a horo-shrinker, the tangency principle asserts that $\Sigma$ and $\Pi$ agree in the largest neighborhood of both surfaces containing the tangency point. This implies that $\Sigma\subset\Pi$, which is a contradiction.
\end{proof}

We finish this section giving the classification of horo-shrinkers invariant by a one-parameter group $\mathcal{H}$ of hyperbolic translations. A hyperbolic translation  of $\h^3$   is an isometry that leaves fixed two points of the ideal boundary $\h_\infty^3$.   In the upper-halfspace model of $\h^3$, and after an isometry, we can assume that these two points are the origin $O$ of $\r^3$ and the infinity. Then   a hyperbolic translation is an   Euclidean  homothety from $O$ and the corresponding  group  is  $\mathcal{H}=\{(x,y,z)\mapsto t(x,y,z):t>0\}$. In particular,  a surface invariant by  $\mathcal{H}$  can be viewed as  a radial graph on the hemisphere $\s^2_{+}=\{(x,y,z)\in\r^3_+:x^2+y^2+z^2=1\}$.

\begin{theorem}\label{t1} Let $\mathcal{H}$ be a one-parameter group of hyperbolic translations fixing two points $p,q\in\h^3_\infty$. Then, totally geodesic planes containing both points are   the only horo-shrinkers invariant by the group $\mathcal{H}$.
\end{theorem}

\begin{proof} 
Without loss of generality, we can assume that $\mathcal{H}$ is the group Euclidean homotheties from the origin $O\in\r^3$. Let $\Sigma$ be a horo-shrinker invariant by $\mathcal{H}$.   Since $\Sigma$ is a radial graph on some domain of $\s^2_+$, a parametrization of $\Sigma$ is
$$\Psi(s,t)=t\alpha(s),\quad s\in I\subset\r,\ t\in\r,$$
where $\alpha\colon I\to\s^2_{+}$ is a curve parametrized by the   Euclidean arc-length. Then $|\alpha(s)|_e=|\alpha'(s)|_e=1$ for all $s\in I$.  The Euclidean mean curvature $H_e$ and the Euclidean unit vector $N^e$ are
$$H_e=\frac12\frac{\langle\alpha'\times \alpha,\alpha''\rangle_e}{t},\quad N^e=\alpha'\times\alpha.$$
Then, \eqref{eq1} is 
$$2\alpha_3 \langle\alpha'\times \alpha,\alpha''\rangle_e +(\alpha'\times\alpha)_3=\frac{\langle\alpha'\times \alpha,\partial_z\rangle_e}{t\alpha_3}=\frac{(\alpha'\times\alpha)_3}{t\alpha_3},$$
where the last equality is because   $\langle\alpha'\times \alpha,\partial_z\rangle_e=(\alpha'\times\alpha)_3$. If we write this equation as 
$$t\alpha_3\left(2\alpha_3 \langle\alpha'\times \alpha,\alpha''\rangle_e +(\alpha'\times\alpha)_3\right)-(\alpha'\times\alpha)_3=0,$$
we have a polynomial equation on $t$. Thus, we deduce $\langle\alpha'\times \alpha,\alpha''\rangle_e=0$ and $(\alpha'\times\alpha)_3=0$ because $\alpha_3\not=0$. Since $\langle\alpha'\times \alpha,\alpha''\rangle_e=0$, and $\alpha$ is a unit speed curve in $\s^2_{+}$, we have $\alpha''=-\alpha$. Thus $\alpha$ is an Euclidean geodesic of $\s^2_+$, that is, a great (hemi) circle of $\s^2_{+}$. Using that $(\alpha'\times\alpha)_3=0$, we deduce that $\alpha$ is included in a vertical plane through $O$. In consequence, $\Sigma$ is  a vertical plane  containing $O$. 
\end{proof}

\begin{remark}
Notice that the same proof is valid for horo-expanders, proving that totally geodesic planes (vertical planes) are the only horo-expanders that are invariant under hyperbolic translations.  This completes the classification given in \cite{mr} of all horo-expanders invariant by a one-parameter group of isometries of $\h^3$.
\end{remark}

 \section{The grim reapers}\label{sec3}

In this section we classify the horo-shrinkers invariant by a one-parameter group $\mathcal{P}$ of parabolic translations. A parabolic  translation of $\h^3$ is an isometry that leaves fixed one double point of the ideal boundary $\h_\infty^3$.  After an isometry of $\h^3$,  we can assume that this point is $\infty\in\h^3_\infty$ and thus a parabolic translation  is simply a horizontal Euclidean  translation. Then the group $\mathcal{P}$ is determined by a horizontal direction $(a,b,0)\in\r^3$ being $\mathcal{P}=\{(x,y,z)\mapsto (x,y,z)+t(a,b,0):t\in\r\}$.  Hence, a surface  invariant by $\mathcal{P}$ is    a ruled surface of $\r^3_{+}$ whose all rulings are horizontal straight-lines parallel to $(a,b,0)$. In analogy with the Euclidean context, we give the following definition.

\begin{definition} A grim reaper   is a horo-shrinker that is invariant by a one-parameter group of parabolic translations. 
\end{definition}

Let $\Sigma$ be a grim reaper. Without loss of generality, we can assume that the rulings of $\Sigma$ are parallel to the direction $(0,1,0)$. A parametrization of $\Sigma$ is 
\begin{equation}\label{pp}
\Psi(s,t)=(x(s),t,z(s)),\quad t\in\r, s\in I\subset\r, 
\end{equation}
where $\alpha(s)=(x(s),0,z(s))$ is a planar curve contained in the $xz$- plane. Suppose that $s$ is   the hyperbolic arc-length, that is $\langle\alpha'(s),\alpha'(s)\rangle=1$. This reads as
$$
\frac{1}{z(s)^2}(x'(s)^2+z'(s)^2)=1,
$$
hence there is a smooth function $\theta=\theta(s)$   such that
$$
x'(s)=z(s)\cos\theta(s),\qquad z'(s)=z(s)\sin\theta(s).
$$
The    unit normal   is   $N=zN^e=(-z',0,x')$.
The Euclidean mean curvature $H_e$ of $\Sigma$ is $H_e=\kappa/2$, where $\kappa$ is the Euclidean curvature of $\alpha$. Since $\kappa=\theta'/z$, then $H_e=\theta'/(2z)$. Using \eqref{mean}, we have  
$$
H=\frac{\theta'}{2}+\frac{x'}{z}.
$$
Since $\langle N,\partial_z\rangle=x'/z^2$, then \eqref{eq1} is 
\begin{equation}\label{eq2}
\frac{\theta'}{2}+\cos\theta=\frac{\cos\theta}{z}.
\end{equation}
Therefore,  the coordinate functions $x(s)$, $z(s)$ and $\theta(s)$ satisfy
\begin{equation}\label{eqs2}
\left\{
\begin{split}
x'(s)&=z(s)\cos\theta(s),\\
z'(s)&=z(s)\sin\theta(s),\\
\theta'(s)&=2\cos\theta(s)\frac{1-z(s)}{z(s)}.
\end{split}\right.
\end{equation}
Since the aim of this section is the geometric description of the grim reapers, we study the shape of the solution curves of \eqref{eqs2}.  First, we see that each solution of \eqref{eqs2} remains at a bounded distance to the plane $z=0$.

\begin{proposition}\label{proporbitaborde}
Let $(x(s),z(s),\theta(s))$ be a solution to \eqref{eqs2}. Then there is $\delta>0$ such that $z(s)\geq\delta$ for all $s\in I$.  
\end{proposition}

\begin{proof}
Multiplying the last equation of \eqref{eqs2} by $\cos\theta\sin\theta$ and taking into account that $z'/z=\sin\theta$, we have 
$$\frac{\sin\theta(\sin\theta)'}{1-\sin^2\theta}=2\sin\theta\frac{1-z}{z}=2\frac{1-z}{z^2}z'.$$
Hence   we deduce that there exists $c\in\r$ such that 
\begin{equation}\label{fi}
\cos\theta=cz^2e^{2/z}.
\end{equation}
Now, arguing by contradiction, assume that there is a sequence $s_n\to s_1$ such that $z(s_n)\rightarrow0$ as $n\rightarrow \infty$, being $s_1$ either finite or infinite. Substituting in \eqref{fi} the right-hand side diverges to $\infty$, which it is not possible because the left-hand side is bounded.
\end{proof}
As a consequence, each solution of \eqref{eqs2} is defined in $\r$ because the functions on the right-hand side of \eqref{eqs2} are bounded. 

Recall that the function $x(s)$ does not explicitly appear in \eqref{eqs2}, but only its derivative. Geometrically this implies that any solution of \eqref{eqs2} remains a solution after a parabolic translation $(x,0,z)\mapsto (x+t,0,z)$. Consequently, in order to study the  properties of the solutions of \eqref{eqs2} it is enough  to consider the  nonlinear autonomous system
\begin{equation}\label{eqs}
\left(
\begin{array}{c}
z'\\
\theta'
\end{array}
\right)=
\left(
\begin{array}{c}
z\sin\theta\\
\displaystyle{2\cos\theta\frac{1-z}{z}}
\end{array}
\right),
\end{equation}
defined in the domain $ \{(z,\theta): z>0, \theta\in\mathbb{R}\}$. Indeed, if we fix $(x_0,z_0,\theta_0),\ z_0>0$, let $(z,\theta)$ be the unique solution to \eqref{eqs2} for the initial data $z(0)=z_0>0$, $\theta(0)=\theta_0$, and define $x$ as the solution to $x'=z\cos\theta,\ x(0)=x_0$. Then, $\alpha(s)=(x(s),0,z(s))$ is the generating curve of  a surface parametrized by \eqref{pp} that is a solution to \eqref{eq1}.

By periodicity of the trigonometric functions, we define the \emph{orbits} as the solutions $\gamma(s)=(z(s),\theta(s))$ of \eqref{eqs}, which are defined for $z>0$ and $\theta\in(-\pi,\pi)$. By uniqueness of the Cauchy problem, two different orbits cannot intersect, hence the $(z,\theta)$-domain  $(0,\infty)\times (-\pi,\pi)$ is foliated by all the orbits.

The following result exhibits that we can reduce the  study of the orbits essentially to $\theta\in(0,\pi/2)$. Its proof follows immediately, hence it is omitted.

\begin{proposition}\label{prfases}
The following properties hold:
\begin{enumerate}
\item If $\gamma(s)=(z(s),\theta(s))$ is an orbit, then $\overline{\gamma}(s)=(z(-s),-\theta(-s))$ is also an orbit. Consequently, every orbit $\gamma$ is symmetric with respect to the line $\theta=0$.
\item If $\gamma(s)=(z(s),\theta(s))$ is an orbit for $\theta\in(-\pi/2,\pi/2)$, then $\overline{\gamma}(s)=(z(s),-\theta(s)+\pi)$ is an orbit for $\theta\in(\pi/2,3\pi/2)$.
\end{enumerate}
\end{proposition}

We define the phase plane of \eqref{eqs} as the set 
$$\Theta=\{(z,\theta)\colon z>0,\theta\in(-\pi/2,\pi/2)\}.$$
 The coordinates $(z,\theta)$ are in one-to-one correspondence to the orbits of \eqref{eqs}. The motion of any orbit in $\Theta$ is uniquely determined by the sign of the functions $z'$ and $\theta'$. From \eqref{eqs},   the signs of $z'$ and $\theta'$  are determined by the signs of $\cos\theta$ and $\sin\theta$ as well as of the function $z-1$. For example, any orbit intersecting the line $z=1$ changes the monotonicity of its second coordinate, attaining a local maximum or minimum. In the remaining of the phase plane, the second coordinate of any orbit is strictly monotonous.

The explicit examples of horo-shrinkers given in  Section \ref{sec2} are now viewed as  trivial solutions of \eqref{eqs2} in the following result.

\begin{proposition}\label{prorbit}
Explicit examples of   orbits of \eqref{eqs} are the following: 
\begin{enumerate}
\item The point $(1,0)$. This orbit  corresponds to the horosphere $\H_1$.  
\item The lines $\theta=\pm\pi/2$. These orbits  correspond to vertical planes (totally geodesic planes) of equation $x=x_0$, where $x_0\in\r$.  The parameter $\theta=\pi/2$ implies that the vertical plane is parametrized with increasing height, and for $\theta=-\pi/2$ the height is decreasing.
\end{enumerate}
\end{proposition}

\begin{figure}[h]
\begin{center}
\includegraphics[width=.35\textwidth]{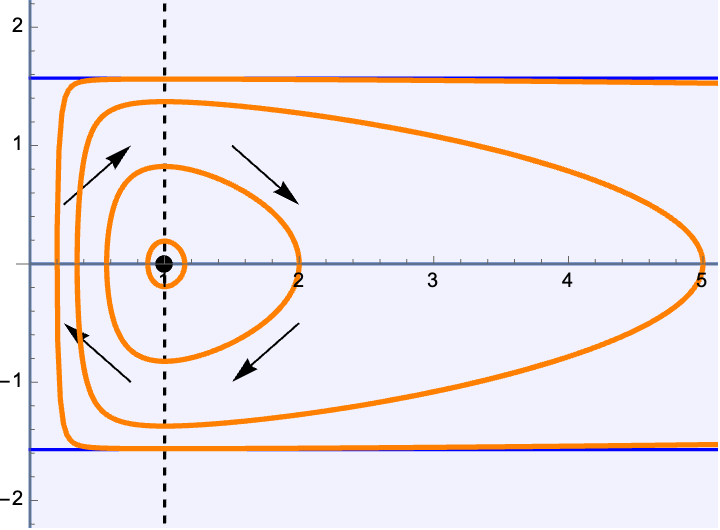}\hspace{1cm}
\includegraphics[width=.45\textwidth]{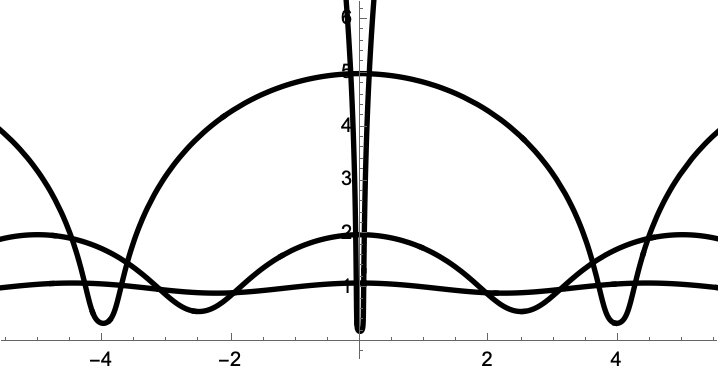}
\end{center}
\caption{Left: the phase plane of \eqref{eqs} and different orbits portrayed. Right: the generating curves of the corresponding orbits. The initial height are $z_0=0.2$, $1.1$, $2$ and $5$.}
\label{fig1}
\end{figure}

As usual, for the description of the orbits of the autonomous system \eqref{eqs}, we analyze its equilibrium points. It is clear that the point $P_0=(1,0)$ is the unique equilibrium. The linearized of the system around $P_0$ is
$$
\left(
\begin{matrix}
0&1\\
-2&0
\end{matrix}
\right).
$$

Since the eigenvalues are imaginary numbers with zero real part, the equilibrium $P_0$ has a center structure. Thus  the orbits of the linearized system are ellipses enclosing $P_0$ in their inner regions. Consequently, the orbits  that are close enough to $P_0$ either spiral around $P_0$ or are closed curves enclosing $P_0$ in their inner regions. However, by Prop. \ref{prfases} the orbits are symmetric about $\theta=0$, hence they cannot spiral around $P_0$. In particular all the orbits stay at a positive distance from $P_0$: see Fig. \ref{fig1}, left. Note: all figures in this paper have been plotted using the software  Mathematica.

We now derive the classification of   the grim reapers.

\begin{theorem}\label{t2}
The classification of the grim reapers is the following:
\begin{enumerate}
\item Vertical planes (totally geodesic planes).
\item The horosphere $\H_1$.
\item A one-parameter family of entire graphs $\mathcal{G}(z_0),\ z_0\in(0,1)$ that are periodic  along the $x$-direction. The value $z_0$ indicates  the Euclidean distance of $\mathcal{G}(z_0)$ at $z=0$. Moreover: 
\begin{enumerate}
\item  If $z_0\rightarrow0$, then $\mathcal{G}(z_0)$ converges to a double covering of a vertical plane.
\item  If $z_0\rightarrow1$ then $\mathcal{G}(z_0)$ converges to the horosphere  $\H_1$.
\item For each $z_0\in(0,1)$ there exists a unique $z_0^*\in(1,\infty)$ that corresponds to the Euclidean height of $\mathcal{G}(z_0)$ at $z=0$. Hence $\mathcal{G}(z_0)$ can be also parametrized in terms of $z_0^*$, being equivalent. In fact, if $z_0\rightarrow0$ (resp. $z_0\rightarrow1$) then $z_0^*\rightarrow\infty$ (resp. $z_0^*\rightarrow1$).
\end{enumerate}
\end{enumerate}
\end{theorem}

\begin{proof}

The first two types of surfaces were already depicted in Prop. \ref{prorbit}. Now, fix some $z_0\in(0,1)$. We call $\mathcal{G}(z_0)$ the grim reaper which is the graph of the solution of \eqref{eqs2} for initial conditions $x(0),z(0),\theta(0))=(0,z_0,0)$. Let $\gamma_{z_0}$ be the orbit with initial data $\gamma_{z_0}(0)=(z_0,0)$ and let $\alpha_{z_0}$ be the corresponding generating curve of the grim reaper. Then $\gamma_{z_0}$ is vertical at $(z_0,0)$, for $s>0$ and by monotonicity both coordinates functions $z(s)$ and $\theta(s)$ of $\gamma_{z_0}$ strictly increase. Since $\gamma_{z_0}$ cannot intersect the orbit $\theta=\pi/2$, necessarily $\gamma_{z_0}$ intersects the vertical line $z=1$ where its $\theta$-coordinate  attains a local maximum. Then, $\theta$ decreases and since $\gamma_{z_0}$ cannot converge to $P_0$ due to its center structure, $\gamma_{z_0}$ intersects again the line $\theta=0$ at some $(z_0^*,0)$, where $z_0^*>1$. Finally, by symmetry of the phase plane with respect to the line $\theta=0$, $\gamma_{z_0}$ closes again at the point $(z_0,0)$. See Fig. \ref{fig1}, left. Note that if $z_0\rightarrow1$ then $\gamma_{z_0}\rightarrow P_0$, while if $z_0\rightarrow0$ then $\gamma_{z_0}$ converges to both $\theta=\pm\pi/2$.

At this point, for each $z_0\in(0,1)$ the point  $z_0^*\in(1,\infty)$ corresponds with a local maximum of the function $z=z(s)$ by Prop. \ref{mm}. It could happen that $(z_0)_n\rightarrow0$ and $(z_0^*)_n\rightarrow (z_0^*)_\infty<\infty$. However, this possibility cannot happen in virtue of Prop. \ref{proporbitaborde}. Indeed, assume by contradiction that this behavior occurs, take some $z^*>(z_0^*)_\infty$ and let $\gamma_{z^*}$ be the orbit passing through $(z^*,0)$ at $s=0$. Then, when $s$ increases $\gamma_{z^*}$ cannot intersect $\theta=-\pi/2$, hence $\gamma_{z^*}$ intersects $z=1$. Since it cannot intersect again the line $\theta=0$ (because it would correspond to some $z_*\in(0,1)$, a contradiction), the only possibility for $\gamma_{z^*}$ is to converge to $z=0$. But this contradicts Prop. \ref{proporbitaborde}. As a consequence, for each $z_0\in(0,1)$ there exists exactly one $z_0^*\in(1,\infty)$, being both intervals in a one-to-one correspondence. 

Therefore, the $z$-coordinate of the associated generating curve $\alpha_{z_0}$ is periodic. Since $\theta\in(-\pi/2,\pi/2)$ we conclude that $x'>0$. This implies that  $x$ is strictly increasing, hence  the curve $\alpha_{z_0}$ is   periodic along the $x$-axis, in particular, invariant under a discrete group of   translations along the $x$-axis. The curve $\alpha_{z_0}$ is a graph on the $x$-axis because $x'>0$. The Euclidean height of $\alpha_{z_0}$ at $z=0$ is $z_0^*$ and its distance to $z=0$ is $z_0$. If $z_0\rightarrow1$ then the curve $\alpha_{z_0}$ converges to the horizontal line $z=1$. If $z_0\rightarrow0$ then $\alpha_{z_0}$ converges to a double covering of a vertical line, since $\gamma_{z_0}$ converges to both $\theta=\pm\pi/2$. See Fig. \ref{fig1}, right. This concludes the proof.
\end{proof}

With the same notation as in the proof of Thm.  \ref{t2}, let us consider initial conditions $(x(0),z(0),\theta(0))=(0,z_1,0)$ with $z_1>1$ in system \eqref{eqs2}, and let $\gamma_{z_1}$ be the corresponding orbit passing through $(z_1,0)$. Then, $\gamma_{z_1}$ passes through some $(z_1^0,0)$, with $z_1^0<1$ being the minimum value of the height function $z$. Definitively, $(z_1^0)^*=z_1$ and therefore up to a parabolic translation orthogonal to the ruling direction $(0,1,0)$, the grim reaper $\mathcal{G}(z_1)$ agrees with $\mathcal{G}((z_1^0)^*)$. 
\begin{corollary} Let $z_0\in (0,1)$. Then there is a unique $z_1\in (1,\infty)$ such that $\mathcal{G}(z_0)$ and $\mathcal{G}(z_1)$ coincide up to a   parabolic translation orthogonal  to $(0,1,0)$. 
\end{corollary}

To finish this section, we address the Dirichlet problem at infinity for Eq. \eqref{mean}. More precisely, let $\Omega\subset \r^2=\{z=0\}\subset\h^3_\infty$ be a bounded domain with smooth boundary. We are asking for functions $u\in C^\infty(\Omega)\cap C^0(\overline{\Omega})$ such that $u$ satisfies \eqref{div} in $\Omega$ with $u>0$ in $\Omega$ and $u=0$ along $\partial\Omega$. This is motivated by the pioneering  works of the theory of constant mean curvature surfaces in hyperbolic space $\h^3$ due to Anderson  ($H=0$) and to Tonegawa ($0<H<1$) \cite{and,to}. The Dirichlet problem at infinity  for horo-expanders was considered in \cite{mr} proving existence under general hypothesis on convexity of $\Omega$. However, for Eq. \eqref{div} we show that the Dirichlet problem at infinity is not solvable. The key is that the grim reapers  of  Thm. \ref{t2} allow us to compare with the possible solutions of the Dirichlet problem at infinity.

\begin{theorem}\label{t22}
There are no solutions of the Dirichlet problem at infinity for Eq.  \eqref{div}.
\end{theorem}

\begin{proof} By contradiction, suppose that $u$ is a solution of \eqref{div}, $u>0$ in $\Omega$ and with initial condition $u=0$ along $\partial\Omega$. Let $\Sigma$ be the graph of $u$ when $u$ is defined in $\Omega$ and let $u_M>0$ be the maximum of $u$ in $\Omega$ which exists because $\overline{\Omega}$ is compact and $u$ is continuous in $\overline{\Omega}$. After a parabolic translation along the $x$-direction, we can assume that $\Sigma$ is included in the half-space $\{(x,y,z)\in\r^3_+:x>0\}$. By Prop. \ref{mm}, let $u_M>1$ be the maximum value of $u$ in $\Omega$. Using Thm. \ref{t2}, let $z_0\in (0,1)$ sufficiently close to $0$ such that $z_0^*$, the maximum height of $\mathcal{G}(z_0)$, satisfies $z_0^*>u_M$. Take the piece of $\mathcal{G}(z_0)$ comprised between two consecutive maximum of $\mathcal{G}(z_0)$. To be precise, if we write the generating curve of $\mathcal{G}(z_0)$ as $z=z(r)$, let $r_0>0$ be such that $z(0)=z_0$, $z(\pm r_0)=z_0^*$ and $r=\pm r_0$ are the only maximum of $z(r)$ in the interval $[-r_0,r_0]$. Consider $\mathcal{G}(z_0)^F$ the piece of $\mathcal{G}(z_0)$ determined by $z(r)$ in the interval $r\in [-r_0,r_0]$, that is, 
$\mathcal{G}(z_0)^F=\mathcal{G}(z_0)\cap \{(x,y,z)\in\r^3: -r_0\leq x\leq r_0\}$. Notice that the boundary of $\mathcal{G}(z_0)^F$ are two straight-lines parallel to the $y$-axis and both situated at height $z_0^*$. 

Let us move $\mathcal{G}(z_0)^F$ by translations along the $x$-direction with $x\searrow-\infty$ until $\mathcal{G}(z_0)^F$ does not intersect $\Sigma$. This is possible because $\Sigma$ is included in the half-space $x>0$ and the rulings of $\mathcal{G}(z_0)^F$ are parallel to the $y$-axis. Next, we move $\mathcal{G}(z_0)^F$ by translations along the $x$-direction with $x\nearrow\infty$ until the first contact point $p$ with $\Sigma$. Let $\widetilde{\mathcal{G}(z_0)^F}$ denote  the position of $\mathcal{G}(z_0)^F$ when it touches $\Sigma$. This point exists because $\overline{\Omega}$ is compact. Since $z_0>0$, then $z(p)\geq z_0$, so it is an interior point of $\Sigma$ (or equivalently, $z(p)\not=0$).   On the other hand,   $z(p)\leq u_M<z_0^*$, so $p$ is an interior point of $\widetilde{\mathcal{G}(z_0)^F}$. Definitively,  $p$ is a common interior point of $\Sigma$ and $\widetilde{\mathcal{G}(z_0)^F}$. The tangency principle implies that $\Sigma$ is included in $\widetilde{\mathcal{G}(z_0)^F}$. This is a contradiction because $\widetilde{\mathcal{G}(z_0)^F}$ is contained in the half-space $\{(x,y,z)\in\r^3_+: z>z_0\}$ where $z_0>0$. 
\end{proof}

\section{Spherical rotational horo-shinkers}\label{sec4}
 
In this section we classify all horo-shrinkers invariant by one-parameter group $\mathcal{S}$ of spherical rotations of $\h^3$.    After an isometry of $\h^3$, we can suppose that the common rotation axis of the elements of $\mathcal{S}$ is   the $z$-axis. Then the elements of $\mathcal{S}$ are simply Euclidean rotations about the $z$-axis, being $\mathcal{S}=\{(x,y,z)\mapsto (x\cos t-y\sin t,x\sin t+y\cos t,z):t\in\r\}$. Thus, a   parametrization   of a spherical rotational surface $\Sigma$ is  
 $$\Psi(s,t)=(x(s)\cos t,x(s)\sin t,z(s)),\quad s\in I\subset\r, t\in\r,$$
 where $\alpha(s)= (x(s),0,z(s))$ is the generating curve.    Notice that horospheres, viewed as horizontal planes of equation $z=c$, $c>0$, also are spherical rotational surfaces. Consequently, the horosphere $\H_1$ is a spherical rotational horo-shrinker.
 
Assume that $\alpha$ is parametrized by   the Euclidean arc-length. Then $\alpha'(s)=( \cos\theta(s),0,\sin\theta(s))$, for some smooth function $\theta=\theta(s)$. The Euclidean mean curvature $H_e$ and the Euclidean unit normal $N^e$ of $\Sigma$ are, respectively, 
 $$H_e=\frac12\left( \theta'+\frac{\sin\theta}{x}\right),\quad N^e=(-\sin\theta\cos t,-\sin\theta\sin t,\cos\theta).$$
By \eqref{mean}, the  equation \eqref{eq1} writes as
\begin{equation}\label{eq3}
\frac{z}{2}\left(\theta'+\frac{\sin\theta}{x}\right)+\cos\theta=\frac{\cos\theta}{z}.
\end{equation}
Thus, Eq. \eqref{eq1} is equivalent to say that the coordinate functions of the curve $\alpha$   satisfy
\begin{equation}\label{spherical}
\left\{
\begin{split}
x'(s)&=\cos\theta(s),\\
z'(s)&=\sin\theta(s),\\
\theta'(s)&=-\frac{\sin\theta(s)}{x(s)}+2\cos\theta(s)\frac{1-z(s)}{z(s)^2}.
\end{split}
\right.
\end{equation}

The study of the spherical rotational horo-shrinkers, or equivalently, of the solutions of \eqref{spherical}, is separated in two cases depending if the surface, or equivalently   the generating curve, intersects or not the rotation axis. 

First consider the case the surface  intersects the rotational axis. We will prove the existence of such surfaces and, in such a case, that this intersection must be orthogonal. Since $\alpha$ intersects the $z$-axis, then at the initial value, say  $s=0$, for \eqref{spherical}, $x(0)$ must be $0$ and $z'(0)=0$. However,  the existence of solutions of \eqref{spherical} is not assured by the standard ODE theory because \eqref{spherical} is degenerate at $x=0$. To such a existence, we  parametrize the  curve $\alpha$ by $r\mapsto (r,0,z(r))$, $z(r)>0$, then   Eq.   \eqref{eq3}  writes as 
 \begin{equation}\label{rot1}
\frac{z}{2}\left(\frac{z''}{(1+z'^2)^{3/2}}+\frac{z'}{r(1+z'^2)^{1/2}}\right)+\frac{1}{(1+z'^2)^{1/2}}=\frac{1}{z(1+z'^2)^{1/2}}.
\end{equation}
Next we prove that there exist solutions of \eqref{rot1}   defined at $r=0$   such that $z'(0)=0$. 
 
\begin{theorem} \label{t3}
If $z_0>0$, then there exist $R>0$ and a solution $z\in C^2([0,R])$ of (\ref{rot1}) with initial conditions
\begin{equation}\label{initial}
z(0)=z_0>0, \quad z'(0)=0.
\end{equation}
\end{theorem}

\begin{proof} Multiplying   \eqref{rot1} by $r$,    we can write \eqref{rot1} as 
\begin{equation}\label{rot-r}
\left({\displaystyle \frac{r z'(r)}{\sqrt{1+z'(r)^2}}}\right)'= {\displaystyle 2r\frac{1-z(r)}{z(r)^2\sqrt{1+z'(r)^2}}}.
\end{equation}
Define the functions  
$$g:\r^+\times\r\rightarrow\r,\quad g(x,y)=\frac{2(1-x)}{x^2\sqrt{1+y^2}},$$
$$\varphi:\r\rightarrow\r,\quad \varphi(y)=\frac{y}{\sqrt{1+y^2}}.$$
From \eqref{rot-r},  a  function $z=z(r)\in C^2([0,R])$ satisfies \eqref{rot1}-\eqref{initial} if and only if   $(r\varphi(z'))'=r g(z,z')$ under initial conditions \eqref{initial}.  The inverse function of $\varphi$ is $\varphi^{-1}(x)=x/\sqrt{1-x^2}$, which is defined in $(-1,1)$. Fix $R>0$ that will be determined later and define the operator ${\mathsf T} :C^1([0,R])\to C^1([0,R])$ by
\begin{equation}\label{tt}
({\mathsf T}z)(r)=z_0+\int_0^r\varphi^{-1}\left(\frac{1}{s}\int_0^s t g(z,z')dt\right)ds.
\end{equation}
It is clear that a fixed point of   ${\mathsf T}$ is a solution of the initial value problem \eqref{initial}-\eqref{rot-r}. First, we prove the existence of $\epsilon>0$ such that  ${\mathsf T}$ is well defined in a closed ball $\overline{B(z_0,\epsilon)}$ of  $C^1([0,R])$. Here we understand that   the space $C^1([0,R])$ is endowed the usual sup-norm $\|z\|=\|z\|_\infty+\|z'\|_\infty$. For this, let $\epsilon>0$ such that $\epsilon<z_0$, and consider $g$ defined in $[z_0-\epsilon,z_0+\epsilon]\times\r$. Let $M>0$ such that $|\frac{2(1-z)}{z^2}|\leq M$ for all $|z-z_0|\leq\epsilon$.  Let  
$R\leq\min\{\frac{1}{M},\frac{\sqrt{3}\epsilon}{2}, \frac{\sqrt{3}\epsilon}{2M}\}$. We have
$$
\int_0^s\frac{t}{s}g(z,z')\, dt\leq \int_0^s\frac{t}{s}M\, dt\leq\frac{RM}2\leq\frac 12,
$$
because $R\leq 1/M$. This allows to apply $\varphi^{-1}$ in the parenthesis of \eqref{tt}.  In order to use  the Banach fixed point theorem, we need the two following steps.  

(1) {\it The map ${\mathsf T}$ satisfies ${\mathsf T}(\overline{B(z_0,\epsilon)})\subset \overline{B(z_0,\epsilon)}$}. To prove this inclusion, let $z\in \overline{B(z_0,\epsilon)}$. By using that $\varphi^{-1}$ is increasing, we have
\begin{equation*}
\begin{split}
|({\mathsf T}z)(r)-z_0|&\leq \int_0^r\varphi^{-1}\left(\int_0^s\frac{t}{s(z_0-\epsilon)}\, dt\right)\, ds <\varphi^{-1}\left(\frac12\right)R=\frac{R}{\sqrt{3}}\leq\frac{\epsilon}{2},\\
|({\mathsf T}z)'(r)|&\leq  \varphi^{-1}\left(\int_0^s\frac{t}{s}M\, dt\right)\leq\varphi^{-1}\left(\frac{R}{2}M\right)=M\frac{R}{\sqrt{4 -RM^2}}\leq \frac{RM}{\sqrt{3}}\leq \frac{\epsilon}{2}
\end{split}
\end{equation*}
because $R\leq \sqrt{3}\epsilon/2$ and $R\leq \sqrt{2}\epsilon/(2M)$, respectively.  As a conclusion, $\|{\mathsf T}z\|\leq\epsilon$.

(2) {\it The map ${\mathsf T}$ is a contraction}.  The functions $g$ and  $\varphi^{-1}$ are  Lipschitz continuous in
$[z_0-\epsilon,z_0+\epsilon]\times[-\epsilon,\epsilon]$ and  $[-\epsilon,\epsilon]$, respectively provided $0<\epsilon<\min\{z_0,1\}$. Let $L=\min\{L_g,L_{\varphi^{-1}}\}$, where $L_g$ and $L_{\varphi^{-1}}$ stand for the Lipschitz constants of $g$ and $\varphi^{-1}$, respectively. Given  $z,\tilde{z}\in\overline{B(z_0,\epsilon)}$,  for all $r\in [0,R]$ we have
\begin{equation*}
\begin{split}
|({\mathsf T}z)(r)-({\mathsf T}\tilde{z})(r)|&\leq  L\int_0^r\frac{1}{s}\int_0^st\left| g(z,z')-g(\tilde{z},\tilde{z}')\right|\, dt\, ds\\
&\leq L^2\int_0^r\frac{1}{s}\int_0^s t \left( \|z-\tilde{z}\|_\infty+\|z'-\tilde{z}'\|_\infty\right)\, dt\, ds\\
&=L^2 \| z-\tilde{z}\|\int_0^r\frac{s}{2}\, ds=\frac{r^2L^2}{4}\|z-\tilde{z}\|.
\end{split}
\end{equation*}
Analogously, 
$$|({\mathsf T}z)'(r)-({\mathsf T}\tilde{z})'(r)|\leq\frac{rL^2}{2}\|z-\tilde{z}\|.$$
Therefore
$$\| {\mathsf T}z-{\mathsf T}\tilde{z}\|\leq\min\{\frac{R^2L^2}{4},\frac{RL^2}{2}\}\|z-\tilde{z}\|.$$
Since $L$ is fixed, by choosing $R>0$ small enough, we conclude that ${\mathsf T}$ is a contraction. 
 
  The solution $z=z(r)$ obtained by the Banach fixed point theorem lies in $C^1([0,R])\cap C^2((0,R])$. We prove  that $z(r)$ can be extended up to  $C^2$-regularity at $r = 0$. From    \eqref{rot1},   the L'H\^{o}pital rule gives 
\begin{equation}\label{hopital}
\lim_{r\rightarrow 0} z''(r)=\frac{1-z_0}{z_0^2}.
\end{equation}
This completes the proof of the theorem.
\end{proof}

Once we have proved the existence of   spherical rotational horo-shrinkers intersecting orthogonally the rotation axis, our next goal is to achieve a full classification of such surfaces. First, we prove the following result which is valid for any solution of \eqref{spherical}.  

\begin{proposition}\label{pr1} 
Let $\alpha(s)=(x(s),0,z(s))$ be a solution of \eqref{spherical}.   If $\alpha$ is  not a graph, then the $x$-coordinate has exactly one critical point which is a minimum. In consequence, if  $\alpha$ intersects the rotation axis then  $\alpha$ is a graph.
\end{proposition}

\begin{proof}
If $\alpha$ is not a graph, then there is a critical point $s_0$ of $x$, $x'(s_0)=0$ with $x(s_0)>0$. From   \eqref{spherical} we have  $\theta'(s_0)=-\sin\theta(s_0)/x(s_0)=\pm\frac{1}{x(s_0)}$ and thus
$$x''(s_0)=\frac{1}{x(s_0)}>0,$$
which yields that $s_0$ is a local minimum of $x(s)$. This proves that $s_0$ must be a local minimum and in such a case, no more critical points of $x(s)$ exist.

If $\alpha$ intersects the rotation axis at $s=0$, then $x(0)=0$ and $x'(0)=1$. If $s_0>0$ is the first critical point of $x(s)$, then $s_0$ would be a local minimum, a contradiction.  
\end{proof}

As a consequence of Prop. \ref{pr1}, the generating curve $\alpha$ of a spherical rotational horo-shrinker that intersects the rotation axis can be globally parametrized by $z=z(r)$.  By Eq. \eqref{rot1}, $z(r)$ is a solution of the initial value problem
\begin{equation}\label{rot2}
\left\{
\begin{split}
 &\frac{z''}{1+z'^2}+\frac{z'}{r}=2\frac{1-z}{z^2}\\
& z(0)=z_0>0,\quad z'(0)=0.
\end{split}
\right.
\end{equation}

Let $J=[0,r_{max})$ stand for the maximal domain of the solutions of \eqref{spherical}, where $r_{\max}\in\r\cup\{\infty\}$. We denote by $\mathcal{B}(z_0)$ the spherical rotational horo-shrinker generated by $z(r)$, whose intersection with the rotation axis occurs at $z=z_0$. The following result exhibits the properties of $\mathcal{B}(z_0)$ and its classification.  Numerical examples are depicted in Fig. \ref{fig2}.

\begin{theorem}\label{t4}
The spherical rotational horo-shinkers interesecting the rotation axis are the surfaces $\mathcal{B}(z_0)$, where the parameter $z_0>0$ indicates the height of the intersection point of the surface with the rotation axis. Each $\mathcal{B}(z_0)$ is an entire graph that oscillates around the horosphere $\H_1$. Furthermore, 
\begin{enumerate}
\item If $z_0=1$, then $\mathcal{B}(1)=\H_1$.
\item If $z_0\in (0,1)$, then $\mathcal{B}(z_0)$ is strictly convex at $r=0$.
\item If $z_0\in (1,\infty)$, then $\mathcal{B}(z_0)$ is strictly concave at $r=0$.
\end{enumerate}
\end{theorem}

\begin{proof}
The case $z_0=1$ follows immediately by just checking that the constant function $z(r)=1$ fulfills \eqref{rot2}. This proves the assertion (1). Suppose now  $z_0\neq1$. Substituting at \eqref{rot2}, we have 
$$
z''(0)=\frac{1-z_0}{z_0}^2.
$$
If $z_0\in(0,1)$ (resp. $z_0\in (1,\infty)$) then $z''(0)>0$, the function $z(r)$ has a local minimum (resp. local maximum) at $r=0$ and $z(r)$ is strictly convex (resp. concave) for $r>0$ small enough.  This proves (2) and (3).

We now prove  that the solutions $z(r)$ of \eqref{rot2} are   entire graphs (that is, $r_{max}=\infty$)  that oscillate  around the horizontal line $z=1$ in the $xz$-plane. We assume  $z_0\in(0,1)$, as the arguments when  $z_0\in(1,\infty)$ are analogous.  The behavior of $z(r)$ will be deduced by proving a series of claims. 
\begin{enumerate}
\item {\it The function $z(r)$ cannot be  a convex graph for $r>0$}. On the contrary, because    $z',z''$ are positive, we have that $z\rightarrow\infty$ as $r\to r_{max}$. The left-hand side of \eqref{rot2} remains always positive, but its right-hand side is eventually negative, a contradiction.

\item {\it The function $z(r)$ cannot fail to be a graph at finite time $r_0>0$}. Arguing by contradiction,   as $r\rightarrow r_0$ it happens $z(r)\rightarrow z(r_0)$, $z'(r)\rightarrow\infty$, $ z''(r)>0$ for $r$ close to $r_0$. Recall that $\lim_{r\rightarrow r_0}z''(r)$ can be either finite or infinite, but in any case it is positive. Taking limits in \eqref{rot2} as $r\rightarrow r_0$ we see that the left-hand side of \eqref{rot2} goes to $\infty$, while its right-hand side is a finite value, a contradiction. 

\item As a consequence, $z$ must change its convexity, which implies $z''(r_1)=0$ and $z''(r_1)<0$ for $r>r_1$ close enough to $r_1$. In particular, $z'(r_1)>0$ and from \eqref{rot2}, we deduce that $z(r_1)=z_1<1$.

\begin{figure}[h]
\begin{center}
\includegraphics[width=.55\textwidth]{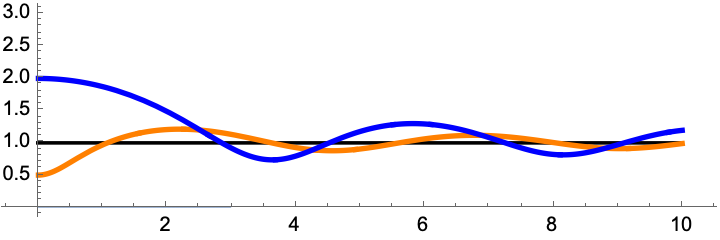}
\end{center}
\caption{Generating curves of spherical rotational horo-shrinkers intersecting the rotational axis, and the horosphere $z=1$ between them. Here $z_0=0,5$, $1$ and $2$.}
\label{fig2}
\end{figure}

\item {\it There are $0<z_m<z_M<\infty$ such that $z_m\leq z(r)\leq z_M$ for all $r\in J$}. Moreover, $z_m=z_0$ (if $z_0>1$, then $z_M=z_0$). In order to prove the claim, let us multiply  \eqref{rot2} by $z'$ and integrate from $0$ to $r$. Then  we obtain
\begin{equation}\label{el}
\frac{1}{2}\log(1+z'^2)+\int_{0}^r\frac{z'(t)^2}{t}dt=-2\left(\frac{1}{z}+\log z\right)+2\left(\frac{1}{z_0}+\log z_0\right).
\end{equation}
If there exists a sequence $r_n\to r_{max}$ such that $z(r_n)\to\infty$, then the right hand-side of \eqref{el} goes $-\infty$, a contradiction because the left hand-side is positive. This proves that $z(r)$ is bounded from above. A similar argument shows that $z(r)$ is bounded from below, by taking a sequence $r_n\rightarrow r_{max}$ such that $z(r_n)\to 0$. 

We now prove that   $z_m=z_0$. On the contrary, let $r_*>0$ be such that $z(r_*)=z_*<z_0$. Letting $r=r_{*}$ in  \eqref{el}, the right hand-side must be positive.  Consider the function $f(t)=-2(\frac{1}{t}+\log t)$, which is negative and increasing in $(0,1)$.  The right hand-side of \eqref{el} writes as $f(z_*)-f(z_0)>0$. Hence  $z_*>z_0$, a contradiction.  

\item {\it We have $r_{max}=\infty$}. We write the first equation of \eqref{rot2} as
\begin{equation}\label{sys}
\left(\begin{array}{l}z\\ z'\end{array}\right)'=\left(\begin{array}{c}z'\\ \displaystyle(1+z'^2)\left(2\frac{1-z}{z^2}-\frac{z'}{r}\right)\end{array}\right).
\end{equation}
Then $r_{max}=\infty$ if we show that the right hand-side of \eqref{sys} is bounded. In fact, by the above claim, it is enough to prove that the function $z'(r)$ is bounded. If there is a sequence $(r_n)\to r_{max}$ such that $|z'(r_n)|\to\infty$, then evaluating \eqref{el} at $r=r_n$ and letting $n\to \infty$, we have that the left hand-side of \eqref{el} diverges. However, the right hand-side is bounded by Claim 4. This contradiction proves the claim. 

\item {\it The function $z(r)$ attains a local maximum at some $r_2>r_1$}. By contradiction, assume that such maximum does not exist, which implies that $z''(r)\leq0$ and $z'(r)>0$ for every $r>r_1$. By the previous claim, since $z(r)$ is strictly increasing and bounded from above, then $z(r)$ has a limit which, without loss of generality, we can suppose that it is  $z_M$. Moreover, $z'(r)\to 0$ as  $r\rightarrow\infty$. 

Next, letting $r\to\infty$  in \eqref{rot2}, we deduce that $z''(r)$ has a limit. Since $z''(r)\leq 0$, we conclude that this limit is $0$. Thus 
$\lim_{r\to\infty}z''(r)=  2\frac{1-z_M}{z_M^2}=0$. This yields $z_M=1$. In particular, the left hand-side of \eqref{rot2} is positive. Therefore, by dividing  in \eqref{rot2} by $z'$ and integrating from $r_1$ to $r$ for $r$ big enough yields
$$
\frac{1}{2}\log\frac{z'^2}{1+z'^2}+\log r+c_1>0,.
$$
for some integration constant $c_1$. After some manipulations we arrive to
$$
z'>\frac{c_2}{\sqrt{r^2-c_2^2}},\quad c_2=e^{-c_1}>0.
$$
Finally, integrating from $r_1$ to $r$ we obtain
$$
z(r)>c_2\,\mathrm{arctanh}\frac{r}{\sqrt{r^2-c_2^2}}+c_3,\quad  \ c_3\in\mathbb{R}.
$$
Letting $r\to\infty$, the right hand-side in this inequality diverges and thus   $z(r)\to \infty$ as $r\to\infty$. This it is not possible by Claim  4. 

After these arguments we ensure the existence of $r_2>r_1$ such that $z'(r_2)=0$. Note that $z(r_2)\neq1$ since otherwise, $\mathcal{B}(z_0)=\H_1$ by uniqueness in \eqref{rot2}: this is not possible because $z_0<1$.   Since $z''(r)\leq0$ for $r<r_2$ close enough to $r_2$ we conclude $z_2>1$. Then   $z''(r_2)<0$, which implies that $z(r)$ attains a local maximum at $r_2$. 

\item From the above claims, we deduce that the function $z(r)$   cannot end being a graph at some finite $r_3>r_2$ because the left-hand side of \eqref{rot2} would be $-\infty$ but the right-hand side is finite. Consequently, $z$ keeps being a graph and for $r>r_2$ small enough we have $z''(r),\ z'(r)<0$. We show that $z$ cannot keep this behavior. Otherwise, for $r\rightarrow r_3>r_2$ we would have $\lim_{r\rightarrow r_3}z(r)=0$ with $z',z''<0$. But this contradicts the Claim 4. Thus  $z(r)$ has to change its curvature, i.e. there exists $r_3>r_2$ such that $z''(r_3)=0$, for which $z(r_3)>1$ since $z'(r_3)<0$. At this point, the only possibilities for $z$ are the following:
\begin{enumerate}
\item $z(r)\rightarrow z_\infty>0$ as $r\rightarrow\infty$;
\item $z'(r_4)=0,\ z''(r_4)>0$ for some $r_4>r_3$.
\end{enumerate}
We prove that the latter is the one that holds. By contradiction, if the former holds, the following would occur
$$
\lim_{r\rightarrow\infty}z'(r)=\lim_{r\rightarrow\infty}z''(r)=0,
$$
which yields a contradiction after substituting in \eqref{rot2}. We conclude that necessarily $z'(r_4)=0$ at some $r_4>r_3$, where by a similar argument as in the case of the maximum we get that $z''(r_4)>0$, i.e. $z$ attains a local minimum at $r_4$, and $z(r_4)<1$ by Prop. \ref{mm}. 

At this point, we have a similar structure as when $z(r)$ started at the rotation axis with an orthogonal intersection at a local minimum of $z(r)$. Therefore, this process is repeated and we see that $z=z(r)$ is an entire graph that oscillates around the horosphere $\H_1$.
\end{enumerate}
 
\end{proof}

The last result of this section is devoted to show the geometric properties of   the spherical rotational horo-shrinkers that do not intersect the rotation axis. See Fig. \ref{fig4}.

\begin{theorem}\label{t5}
Let $\Sigma$ be a spherical rotational horo-shrinker about the $z$-axis such that $\Sigma$ does not intersect the rotation axis. Then $\Sigma$ belongs to a two-parameter family of spherical rotational horo-shrinkers, $\mathcal{W}(x_0,z_0)$, where the parameter $x_0\in (0,\infty)$ indicates the Euclidean distance of $\mathcal{W}(x_0,z_0)$ to the $z$-axis and $z_0$ is the Euclidean distance to $z=0$. Moreover:
\begin{enumerate}
\item The surfaces $\mathcal{W}(x_0,z_0)$ are bi-graphs on the $xy$-plane. 
\item Each $\mathcal{W}(x_0,z_0)$ is contained in the closure of the non-bounded domain determined by the  Euclidean cylinder about the $z$-axis and of radius $x_0$. 
\item Each $\mathcal{W}(x_0,z_0)$ has the topology of an annulus, and its ends oscillate around the horosphere $\H_1$.
\end{enumerate}
\end{theorem}

\begin{proof}
Let us write Eq. \eqref{eq3} considering that the generating curve $\alpha$ is a graph $x=x(r)$ on the $z$-axis, $z>0$. Then $x(r)$ satisfies
\begin{equation}\label{rot3}
x''=\frac{(r^2+2(r-1)xx')(1+x'^2)}{r^2x^2}.
\end{equation}
For $x_0,z_0>0$, let $x=x(r)$ be the solution of \eqref{rot3} with initial conditions  $x(z_0)=x_0$, $x'(z_0)=0$. Since $x''(z_0)=1/x_0$, the function $x(r)$ is   strictly convex locally around $r_0$. This curve generates a spherical rotational horo-shrinker which will be denoted by $\mathcal{W}(x_0,z_0)$. Then  $\mathcal{W}(x_0,z_0)$ starts as a  bi-graph over the $xy$-plane around $(x_0,z_0)$. Let  $\mathcal{W}_+(x_0,z_0)$ denote the  upper graphical component, which is the graph of a function $z=z_+(r)$. Similarly, its lower graphical component is denoted by $\mathcal{W}_-(x_0,z_0)$ and it is the graph of a function  $z=z_-(r)$.

\begin{figure}[h]
\begin{center}
\includegraphics[width=.4\textwidth]{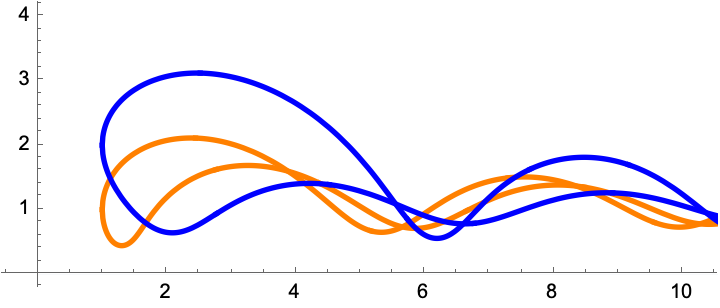}
\end{center}
\caption{Generating curves of spherical rotational horo-shrinkers which do no intersect the rotational axis. In orange, the initial condition is $(x_0,z_0)=(1,1)$. In blue, the initial condition is $(x_0,z_0)=(1,2)$.}
\label{fig4}
\end{figure}

Let us analyze the behavior of $\mathcal{W}_+(x_0,z_0)$, that is, of the function $z_{+}(r)$: for $z_{-}(r)$ the arguments   are analogous. The function $z_+$ satisfies $z_+(x_0)=z_0$, $z_+'(x_0)=\infty$ and for $r>x_0$ close enough to $x_0$ we have $z_+'(r)>0$ and $z_+''(r)<0$. At this point, we follow similar ideas that the ones developed in the proof of Thm. \ref{t4}, so the graph $z_+$ must attain a local maximum, decrease, change its curvature and then attain a local minimum. This process is repeated proving that $z_+(r)$ oscillates around $z=1$ as $r\to\infty$.  
\end{proof}

\begin{remark}
The surfaces $\mathcal{B}(z_0)$ and $\mathcal{W}(x_0,z_0)$ of  Thms. \ref{t4} and \ref{t5} can be thought as the analogous to the bowl soliton and the wing-like examples of translators in the theory of the mean curvature flow in $\r^3$. Up to a translation of $\r^3$, the bowl soliton is the unique translator in $\r^3$ intersecting orthogonally the rotation axis, while the wing-like examples form  a one-parameter family of annuli, parametrized in terms of the distance to the rotation axis. In contrast, the situation for horo-shrinkers is a bit different. If the rotation axis is the $z$-axis (such as it has been considered in this section),   the hyperbolic translations of $\h^3$ from the origin $O$ (Euclidean homotheties)   do not preserve  equation \eqref{eq1}. Therefore, two horo-shrinkers $\mathcal{B}(z_0)$ and $\mathcal{B}(z_1)$, $z_0\not=z_1$, do not coincide by a hyperbolic translation of $\h^3$. For this reason, the family of surfaces $\mathcal{B}(z_0)$ is one-parametric. Similarly, the family   $\mathcal{W}(x_0,z_0)$ is two-parametric.  
\end{remark}

We end this paper with the following observation. Using  Mathematica, it is possible to observe that the surfaces $\mathcal{B}(z_0)$ and $\mathcal{W}(x_0,z_0)$ not only oscillate around $\H_1$ but they converge to it at infinity. However,   the authors have not been able to prove this convergence. The difficulty is that if we project the  system \eqref{spherical} on the $(z,\theta)$-plane, the $2$-dimensional system is not autonomous by the presence of $x$. Or equivalently, the system  \eqref{sys}  is  non-autonomous. Anyway, it is important to point out that $(z,\theta)=(1,0)$ (resp. $(z,z')=(1,0)$) is an equilibrium point of  \eqref{spherical} (resp. \eqref{sys})    regardless of the value of $x$  (resp. of $r$). This equilibrium point corresponds to the horosphere $\H_1$.

\end{document}